\documentclass{amsart}
\usepackage{graphicx}
\usepackage{amsfonts,amsmath,amssymb,eucal}
\usepackage[all]{xy}
\usepackage{amsthm}
\usepackage{amsfonts}
\usepackage{amsmath}
\usepackage{graphicx}
\usepackage{url}

\newtheorem{thm}{Theorem}[section]
\newtheorem{lem}[thm]{Lemma}
\newtheorem{pro}[thm]{Proposition}
\newtheorem{cor}[thm]{Corollary}

\theoremstyle{defn}

\begin{document}
\title{On the cyclic 3-manifold covers of the type Surface $\times$ $\mathbb{R}$}
\author{Jordan Sahattchieve}
\maketitle

\begin{abstract}
This article contains a proof of the fact that, under certain mild technical conditions, the action of the automorphism group of a cyclic 3-manifold cover of the type $S\times\mathbb{R}$, where $S$ is a compact surface, yields a compact quotient.  This result is then immediately applied to extend a theorem in \cite{Moon} on the fiberings over $\mathbb{S}^1$ of certain compact 3-manifolds which are torus sums.  As a corollary, I prove the validity of the conditional main theorem in \cite{Sah} and \cite{SahErr}.  This paper also furnishes a proof of the irreducibility of the summands of compact 3-manifolds which are torus sums and irreducible.
\centering{\textit{In memoriam G. Peter Scott, November 27, 1944 - September 19, 2023}}
\end{abstract}

\begin{section}{Introduction}
Initially, the purpose of this article was to provide proofs of certain mathematical claims in \cite{Moon}, which I discovered were missing - this being necessary for the completeness of the arguments justifying the claims in \cite{Sah} and \cite{SahErr}.  However, it was recently made apparent to me that in light of the recent seminal work by Agol, Wise, and the ensuing results due to Przytyzcki, and Hagen, who prove, among other things, that large classes of manifolds fiber over the circle, the results in \cite{Elkalla}, \cite{Moon}, and \cite{Sah} need to be put into proper context motivating the work contained herein.  Therefore, I will attempt to provide some background on relevant results in the area, with the understanding that this is not meant to be a comprehensive survey on the subject.  For such a survey of recent as well as classical results in 3-manifold topology, see \cite{AFW}.
	
Arguably, the first fibration theorem for compact 3-manifolds was due to Stallings, see \cite{Stallings}:
	
\begin{thm}\label{StallingsT}\textbf{(Stallings, 1961)} Let $M$ be a compact irreducible 3-manifold.  Suppose that there is a surjective homomorphism $\pi_1(M)\rightarrow\mathbb{Z}$ whose kernel $G\neq 1$ is finitely generated and not of order $2$.  Then, $M$ fibers over $\mathbb{S}$ with fiber a compact surface $F$ whose fundamental group is isomorphic to $G$.
\end{thm}
	
There are two important things to take note of in Theorem \ref{StallingsT}: (a) it provides sufficient conditions on the fundamental group of the compact and irreducible manifold $M$, under which $M$ fibers over $\mathbb{S}$, and (b) it establishes the existence of a particular fibering of $M$ - the one whose fundamental group is $G$.  One might feel that, in a certain sense, this is the right fibering of any compact manifold whose fundamental group possesses a normal subgroup $G$, which is motivated by the well-known fact that for any compact 3-manifold which is a fiber fundle over $\mathbb{S}$ with fiber a comapct surface $F$, we have $G=\pi_1(F)\triangleleft\pi_1(M)$.  It is also well-known that a manifold may fiber over the circle in infinitely many ways.  In this context, we note that Theorem \ref{StallingsT} yields a particular fibering having the right - so to speak - fundamental group carried by the fiber surface detected in the fundamental group of $M$.
	
We further note that there are no conditions on the topology or geometry of $M$ in the hypothesis of Theorem \ref{StallingsT} other than compactness and irreducibility.  The reader may be interested to learn that Theorem 1.5 of \cite{Elkalla} yields, as an almost immediate consequence, the fact that, under the hypothesis of Theorem \ref{StallingsT}, the Poincare associate of $M$, $\widehat{M}$, must be irreducible - see Theorem 5.3(a) in \cite{Sah}, \cite{SahErr} and Proposition \ref{irredOfMh} below.
	
The second major result about fiberings over the circle is a theorem of Hempel and Jaco:
	
\begin{thm}\label{Hemel-Jaco}\textbf{(Hempel, Jaco, 1972)}  Let $M$ be a compact, possibly nonorientable manifold, which contains no 2-sided projective plane.  Suppose that there is an exact sequence $1\longrightarrow N\longrightarrow \pi_1(M) \longrightarrow Q \longrightarrow 1$, where $N\neq 1$ is finitely presented, and the quotient $Q$ is infinite.  If $N\neq\mathbb{Z}$, then $\widehat{M}=M_1\#\Sigma$, where $\Sigma$ is a homotopy 3-sphere, and $M_1$ is either: (i) a fiber bundle over $\mathbb{S}$ with fiber a compact surface $F$, or (ii) the union of two twisted $I$-bundles over a compact 2-manifold $F$ which meet in the corresponding $0$-sphere bundles.
		
In either case, $N$ is a subgroup of finite index of $\pi_1(F)$ and $Q$ is an extension of a finite group by either $\mathbb{Z}$ (case (i)) or $\mathbb{Z}_1*\mathbb{Z}_2$ (case (ii)).
\end{thm}
	
Theorem \ref{Hemel-Jaco} long predated both the proof of the Geometrization Theorem by Perelman and a result by Scott proving the coherence of 3-manifold groups.  Because of these theorems, we can now replace \textit{finitely presented} by \textit{finitely generated}, and omit any mention of the homotopy 3-sphere $\Sigma$, as no such, other than $\mathbb{S}^3$, exists.  We again notice that $N$ picks out a preferred fibering of $M$ - a fibering for which the fundamental group of the fiber contains $N$ as a subgroup of finite index.
	
In \cite{Moon}, Moon replaces the requirements that $N$ be finitely generated and that there exist a sujection onto an infinite group with the hypothesis that $N$ is contained in a finitely generated subgroup $U$ of infinite index in $\pi_1(M)$, and that $M$ is geometric, while, in \cite{Elkalla}, Elkalla relaxes also the assumption of normality of $N$ to subnormality plus a certain residual finiteness condition on $\pi_1(M)$, but without any assumptions on the topology or geometry of $M$ other than compactness, connectedness, and $P^2$-irreducibility.  The conclusion in both \cite{Moon} and \cite{Elkalla} is that a finite cover of $M$ fibers over the cirlce and, further, that for the fibration produced, the fundamental group of the fiber is commensurable with $U$.  Thus, the results mentioned so far identify a particular fibering of $M$ - a fibering in which we have certain control over the fiber.
	
In \cite{Thurston1}, Thurston posed 24 research problems setting directions for future work.  Problem 18 in the list conjectures that every hyperbolic manifold has a finite cover which fibers over the circle.  This conjecture became a theorem for compact manifolds in 2012 due to deep results of Agol and Wise.  I refer here to two hard theorems: Agol's fibering criterion, and Wise's results on embedding the fundamental group of a closed hyperbolic manifold into a right-angled Artin group, which I list below:
	
\begin{thm}(Agol, 2008)  Let $M$ be a compact irreducible orientable 3-manifold with $\chi(M)=0$.  If $\pi_1(M)$ is a subgroup of a right-angle Coxeter group of a right-angled Artin group, then $M$ virtually fibers.
\end{thm}
	
\begin{thm}(Wise, 2012)  If $M$ is a non-closed hyperbolic manifold of finite volume, then $\pi_1(M)$ is virtually compact special.
\end{thm}
	
\begin{thm}(Wise, 2012)  If $M$ is a closed hyperbolic manifold that contains a geometrically finite surface, then $\pi_1(M)$ is virtually compact special.
\end{thm}
	
These results coupled with the theorem of Kahn-Markovic in \cite{KahnMarkovic} showing the existence of geometrically finite surfaces in closed hyperbolic manifolds, and the results of Wise on embedding virtually compact special groups into right-angled Artin groups, yield a fibering theorem for hyperbolic 3-manifolds.  More results along these lines have been obtained proving that other classes of manifolds, such as graph manifolds with boundary, see \cite{Przy-Wise}, or closed nonpositively curved graph manifolds, see \cite{Svetlov} , virtually fiber over the cirlce, and the interested reader is again referred to \cite{AFW}.  It is, however, important to note that not all compact 3-manifolds have finite covers which fiber over the circle.  Apart from the trivial examples which are finitely covered by $\mathbb{S}^3$, we know that certain easily constructed closed graph manifolds do not virtually fiber over the circle.  The simplest such example is the torus sum of two trivial hyperbolic surface bundles over $\mathbb{S}^1$, where the gluing map does not respect certain natural bases for the first homology groups of the boundary tori, for more details see \cite{Leeb} and \cite{Svetlov}.  One can easily build closed graph manifolds which do not virtually fiber over the circle and have an arbitrary number of JSJ components, again see \cite{Leeb}.  However, in light of \cite{Leeb} and the combined result, due to Agol, Liu, Przytycki, and Wise, listed as Theorem 5.27 in \cite{AFW}, any irreducible manifold which does not virtually fiber over the cirlce must, in fact, be a closed graph manifold - not a nonpositively curved such, or is finitely covered by $\mathbb{S}^3$.
	
On the other hand, in light of Theorem 1.5 in \cite{Elkalla}, the Stallings type theorems in \cite{HempelJaco}, \cite{Elkalla}, and \cite{Sah}, as well as Stallings' own Theorem \ref{StallingsT}, all show that a compact 3-manifold virtually fibers over the circle under a hypothesis on the fundamental group only, without the conditions of asphericity, a vanishing Euler number, or ruling out a class of manifolds altogether. 
	
I now come to the immediate motivation for the work below.  The main theorem in \cite{Sah} and its corrigendum \cite{SahErr} extends the work in \cite{Moon} in light of the positive solution to the geometrization conjecture by Perelman.  Unfortunately, \cite{Moon} contains two rather significant gaps: (i) its main Theorem 2.10, while stated for compact manifolds, was only proved for closed such, and (ii), even in the case of closed manifolds, the claim in Theorems 2.9 and 2.10 which identifies the fundamental group of the fiber surface has not been proved.
	
Because in \cite{Sah} and its corrigendum I used Moon's results for torus sums to make an inductive argument using the Geometrization Theorem of Perelman, it was essential to have a correct proof for the case of compact manifolds with nonempty boundary.  Also, due to the aforementioned reasons concerning the importance of the claim in \cite{Moon} concerning the fundamental group of the fiber, especially in light of the powerful results arising from the work of Agol and Wise, I have also supplied proof of this correct but unsubstantiated claim made by Moon.  Finally, while I understand that the technical result in Theorem \ref{MainThm} is not specific to dimension 3, I have no need for its higher dimensional verbatim generalizations, and have therefore left it stated for 3-manifolds only, also due to the relative importance of its immediate corollary - Theorem \ref{BoundaryCasesSettled}, which is a result about 3-manifolds only.
	
\end{section}

\begin{section}{The Main Result}
\begin{thm}\label{MainThm}Let $M$ be the $3$-manifold $S\times\mathbb{R}$, which is the trivial line bundle over the compact surface $S$ with or without boundary, and let $M$ be orientable.  Suppose that there is a covering action of the group $G$ on $M$ and a Riemannian metric on $M$ for which this action is by isometries, and suppose that $G$ contains a subgroup $\left\langle \gamma\right\rangle <G$, with $\gamma\in G$, isomorphic to the infinite cyclic group $\mathbb{Z}$.  If $\gamma^n S_0\cap S_0=\emptyset$, for all $n>n_0$, where $S_0$ is any horizontal section $S_0=S\times\left\lbrace t_0\right\rbrace$ with $t_0\in\mathbb{R}$, and $n_0$ is a positive integer, then the quotient of $M$ by the action of $G$ is compact. 
\end{thm}
\begin{proof}Let $g=\gamma^{n_0+1}\in G$.  First, note that $g^i S_0\cap S_0=\emptyset$, for all $i>0$.  Note, also, that if $g^{-i}S_0\cap S_0\neq\emptyset$, for some $i>0$, then $g^{i}(g^{-i}S_0\cap S_0)\neq\emptyset$, and therefore $S_0\cap g^i S_0\neq\emptyset$, hence $i=0$.  Thus, we conclude that $g^i S_0\cap S_0=\emptyset$, for all $i\neq 0$.
	
We will show that the translates $g^iS_0$ partition $M$ into compact connected submanifolds which meet only at their boundaries.  To do this, note that since $S_0$ separates $M$ into two connected components, so does $gS_0$.  Let the closures of the two connected components of $M-S_0$ be $V^+$ and $V^-$.  Then, the closures of the connected components of $M-gS_0$ are $gV^+$ and $gV^-$, respectively.  Since $S_0\cap gS_0=\emptyset$, we must have $gS_0\subset (V^+)^0$ or $gS_0\subset (V^-)^0$.  Suppose, without loss of generality, that $gS_0\subset (V^+)^0$.  Since $S_0$ disconnects $M$ into exactly two connected components, $gS_0$ must disconnect $V^+$ into exactly two connected components also, since $M-gS_0=(V^+-gS_0)\cup V^-$ and $S_0\cap gS_0=\emptyset$.  One of the two connected components of $M-gS_0$ must contain $S_0$.  Let the closure of this connected component be denoted by $D$; it is now clear that $\partial D=S_0\cup gS_0$.  We will show that $D$ is a connected submanifold of $M$ whose boundary is $S_0\cup gS_0$, that $D$ is compact, that $g^i D\cap g^jD=\emptyset$ for $i\neq j$, unless $|i-j|=1$ in which case $g^iD\cap g^j D$ is the translate of $S_0$ which is the common boundary surface of both $g^iD$ and $g^jD$, and that $M=\bigcup g^k D$:

\textit{$D$ is compact:}  Let $h:M\rightarrow\mathbb{R}$ be the projection to the $\mathbb{R}$ factor.  If necessary, we can compose $h$ with a linear map, so that $h(x)>t_0$ for $x\in gS_0$ and $h(V^+)\subset\mathbb{R}_{\geq 0}$.  Since $gS_0$ is compact and connected, $h(gS_0)\subset\mathbb{R}$ is a closed interval; let $h_{max}$ be its upper bound. Let $x_1=(s_1,t_1), x_2=(s_2,t_2)\in D$ be any two points.  Consider the vertical lines $\left\lbrace s_1\right\rbrace\times\mathbb{R}$ and $\left\lbrace s_2\right\rbrace\times\mathbb{R}$.  Because $gS_0$ disconnects $V^+$, $\left\lbrace s_i\right\rbrace\times\mathbb{R}$ must intersect $gS_0$ in at least one point which is at a distance at most $h_{max}-t_0$ away from $x_i$.  Therefore, $x_1$ can be joined to $x_2$ by a path of length no more than $2(h_{max}-t_0)+diam(gS_0)$.  Since $gS_0$ is compact, this number is not $\infty$, and we conclude that $D$ is a closed bounded set.  By the theorem of Hopf-Rinow, see Theorem 2.8(b) in \cite{MandoCRiemGeom}, $D$ is compact.

\textit{$g^2S_0\subset V^+-D$:}  We must have either $g^2S_0\subset V^+-D$, or $g^2S_0\subset D$, or $g^2S_0\subset V^-$.  Let $W^+$ be the closure of the connected component of $M-gS_0$ which contains $D$, and let $W^-$ be the closure of its complement.  If $g^2S_0\subset D$, then $gW^+\subset W^+$ or $gW^+\subset V^+$; however, since $(W^+)^0\supset S_0$, hence $(gW^+)^0\supset gS_0$, thus if $gW^+\subset W^+$, we would have $gS_0\subset (W^+)^0$, which is not true.  Therefore, we have $gW^+\subset V^+$.  On the other hand, $gV^+=W^+$ or $gV^+=W^-$.  Since $D=V^+\cap W^+$, we have $gD=gV^+\cap gW^+$, hence must have $gD=W^+\cap gW^+$, or $gD=W^-\cap gW^+$.  Now, $W^-\cap g^2S_0=\emptyset$, thus $W^-$ is contained in one of $(gW^+)^0$ or $(gW^-)^0$.  Since $(W^+)^0\supset S_0$, we have $(gW^+)^0\supset gS_0$, and $(gW^+)^0\cap W^-\neq\emptyset$, hence $W^-\subset gW^+$.  This imples that if $gD=W^-\cap gW^+$, then $gD$ would fail to be compact since $gD\supset W^-$, which is not true.  Therefore, $gD=W^+\cap gW^+$, in which case we also have $gD\subset D$, as $gW^+\subset V^+$.  We must also have $gD\neq D$, since $D$ and $gD$ have different boundaries, i.e. $gD$ is properly contained in $D$, with $D-gD$ having nonempty interior.  Since $M$ is orientable, the Riemannian metric for which $G$ acts by isometries induces a nowhere vanishing volume form $\omega$ on $M$, such that $G$ preserves the Riemannian measure induced by $\omega$.  But since $gD\subset D$, we must have $vol(gD)<vol(D)$, which is impossible as we must have $vol(gD)=vol(D)$.

Similarly, if $g^2S_0\subset V^-$, let $D'$ be the closure of the connected component of $V^--g^2S_0$ which contains $S_0$: it is a connected submanifold of $M$ whose boundary is $S_0\cup g^2S_0$.  Let $W^+$ be the component of $M-gS_0$ which contains $D\cup D'$, and similarly let $Q^+$ be the component of $M-g^2S_0$ which contains $D\cup D'$.  With this notation, we have $W^+=D\cup V^-$, and $Q^+=D'\cup V^+$.  It is now straightforward to verify that $D\cup D'=W^+\cap Q^+$, and therefore $g(D\cup D')=gW^+\cap gQ^+$.  Since the isometry $x\rightarrow g^{-1}x$ takes $g^2S_0$ to $gS_0$ and $gS_0$ to $S_0$, we must have (a.) $gW^+=V^+$ or $gW^+=V^-$, and (b.) $gQ^+=W^-\subset (V^+)^0$ or $gQ^+=W^+$.  In the cases (1.) $gW^+=V^+$ and $gQ^+\subset V^+$, and (2.) $gW^+=V^-$ and $gQ^+=W^+$, the intersection $gW^+\cap gQ^+$ is noncompact, and therefore cannot equal $g(D\cup D')$.  In the case (3.) $gW^+=V^-$ and $gQ^+\subset (V^+)^0$, the intersection $gW^+\cap gQ^+$ is empty, while case (4.) $gW^+=V^+$ and $gQ^+=W^+$, $g(D\cup D')=D$ is seen to be impossible by a volume comparison argument as above.  Therefore, $g^2S_0\subset V^+-D$.  Now, we see that $D$ and $gD$ share the boundary surface $gS_0$, and if $D^0\cap (gD)^0\neq\emptyset$, then $D-gS_0$ will remain connected while having nontrivial intersection with both of the connected components of $M-gS_0$, which is impossible.

\textit{Let $D_k=\bigcup_{i=0}^{k}g^iD_k$.  Then, $D_k$ is connected, $D_k\subset V^+$, $g^{k+1}S_0\subset V^+-D_k$, and $g^{k+1}D\cap D_k=g^kS_0$.}  We proceed by induction on $k$ to show that $g^{k+1}S_0$ cannot be contained in either $V^-$ or $D_{k-1}$.

Let $W^+$ be the closure of the connected component of $M-g^kS_0$ which contains $D_{k-1}$.  Since $g^{k+1}S_0$ is disjoint from all $g^iS_0$, for $0\leq i \leq k$, if $g^{k+1}S_0\subset D_{k-1}$, then $g^{k+1}S_0\subset g^jD^0$, for some $0\leq j \leq k-1$.  By multiplying by $g^{-j}$, this is equivalent to $g^{k-j+1}S_0\subset D^0$, and since $D^0\subset D_{k-j}$, we have $g^{k-j+1}S_0\subset D_{k-j}$, which is ruled out by the induction hypothesis if $j> 0$.  Therefore, we now assume that $j=0$, hence that $g^{k+1}S_0\subset D^0$.  In this case, $g^{k+1}S_0$ separates $D$ into two connected components, the closures of which we denote by $E$ and $F$.  To be specific, $E$ is the closure of the connected component of $D-g^{k+1}S_0$ bounded by $S_0$ and $g^{k+1}S_0$, while $F$ is the closure of the component bounded by $gS_0$ and $g^{k+1}S_0$.  Since $gW^+$ is the closure of a connected component of $M-g^{k+1}S_0$, we immediately see that either $gW^+=E\cup V^-$ or $gW^+=cl(V^+-E)$.  On the other hand, $gV^+$ is the closure of a connected component of $M-gS_0$, thus we must have $gV^+=D\cup V^-$ or $gV^+=cl(V^+-D)$.  Since $D_{k-1}=V^+\cap W^+$, we have $gD_{k-1}=gV^+\cap gW^+$.  In the cases $gW^+=E\cup V^-$, $gV^+=D\cup V^-$, and $gW^+=cl(V^+-E)$, $gV^+=cl(V^+-D)$, the intersection $gV^+\cap gW^+$ fails to be compact, hence cannot equal $gD_{k-1}$.  In the case $gW^+=E\cup V^-$ and $gV^+=cl(V^+-D)$, we have $gV^+\cap gW^+=\emptyset$, which is again a contradiction.  In the remaining case, $gW^+=cl(V^+-E)$ and $gV^+=D\cup V^-$, we have $gD_{k-1}=F\subset D$, which is ruled out by a volume argument.

Now, if $g^{k+1}S_0\subset V^-$, let $D'$ be the closure of the component of $V^--g^{k+1}S_0$ which contains $S_0$ - it is a connected submanifold bounded by $S_0$ and $g^{k+1}S_0$.  Let $W^+$ be as above, and let $Q^+$ now be the connected component of $M-g^{k+1}S_0$ which contains $D_{k-1}\cup D'$.  We now have $D_{k-1}\cup D'=W^+\cap Q^+$, and therefore\\ $g^{-1}(D_{k-1}\cup D')=g^{-1}W^+\cap g^{-1}Q^+$.  Using the same arguments as before, we see that either $g^{-1}W^+=D_{k-2}\cup V^-$, or $g^{-1}W^+=W^-\cup g^{k-1}D$, and that we also have either $g^{-1}Q^+=W^+$, or $g^{-1}Q^+=W^-$.  In either of the two cases $g^{-1}W^+=D_{k-2}\cup V^-$, $g^{-1}Q^+=W^+$, and $g^{-1}W^+=W^-\cup g^{k-1}D$, $g^{-1}Q^+=W^-$, the intersection $g^{-1}W^+\cap g^{-1} Q^+$ fails to be compact, as it contains the noncompact set $V^-$, or $W^-$, respectively.  In the case $g^{-1}W^+=W^-\cup g^{k-1}D$, $g^{-1}Q^+=W^-$, we have $g^{-1}W^+\cap g^{-1} Q^+=\emptyset$, which is a contradiction.  The last case to consider is $g^{-1}W^+=W^-\cup g^{k-1}D$ and $g^{-1}Q^+=W^+$.  Here, $g^{-1}W^+\cap g^{-1} Q^+=g^{k-1}D$, which gives $g^{-1}(D_{k-1}\cup D')=g^{k-1}D$; a contradiction follows from a volume comparison argument, after noting that $k>0$.

Thus, we conclude that $g^{k+1}S_0\subset V^+-D_{k-1}$, and also that the connected submanifold $g^kD$ is contained in $V^+$, hence $D_k=D_{k-1}\cup g^k D\subset V^+$.

\textit{Given any $x_0\in V^+$, there exists $k$ such that $x\in D_k$.}\\ Let $l=inf\left\lbrace d(x,y): x\in S_0, y\in gS_0\right\rbrace$ be the distance between $S_0$  and $gS_0$ in the path metric induced by the Riemannian metric on $M$.  By construction of the submanifold $D_k$, we see that any path from $g^kS_0$ to $S_0$ must pass through all $g^iS_0$ for $1\leq i\leq k-1$, hence $l_k=\inf\left\lbrace d(x,y): x\in g^kS_0, y\in S_0\right\rbrace=kl$ is an increasing function of $k$.  Suppose now that $x_0\notin D_k$ for all $k$, and choose $k_0$ large enough so that $l_{k_0}=k_0l>d(x_0,S_0)$.  Since $k_0\notin D_{k_0}$, any path from $x_0$ must pass through all the $g^iS_0$ for $1\leq i\leq k$ as all of them disconnect $M$.  Then, we must have $d(x_0,S_0)\geq d(g^{k_0}S_0,S_0)=k_0l$ which is a contradiction.

\textit{$g^{-1}S_0\subset V^-$:}  We argue by contradiction.  If $g^{-1}S_0\subset V^+$, then since $V^+=\bigcup_{i\geq 0}g^iD$, and since the $g^jS_0$ are disjoint for any two different values of $j$, we must have $g^{-1}S_0\subset (g^kD)^0$ for some $k\geq 0$.  Then, $S_0\subset (g^{k+1}D)^0$, where $k+1\geq 1$, thus $g^{k+1}S_0$ and $g^{k+2}S_0$ would lie in two different connected components of $M-S_0$, whereas we know that this is not the case, as $g^iS_0\subset (V^+)^0$ for all $i>0$.

\textit{By an analogous argument, $V^-=\bigcup_{i<0}g^iD$.}  This shows the cocompactness of the action of $G$ on $M$.
\end{proof}

I shall now apply Theorem \ref{MainThm} to prove a result which strengthens Theorem 2.9 in \cite{Moon}, and which by doing so closes the gap in the proof of Theorem 2.10 in \cite{Moon} for manifolds with nonempty boundary.  I shall also prove that a suitably chosen finite cover of $M$ fibers in the desired way, meaning that $\pi_1(F)$ is commensurable with $U$.  This proof covers the case $\partial M\neq\emptyset$ as well as the case $\partial M=\emptyset$, neither of which has been established in \cite{Moon}.  While this does not prove commensurability for the particular cover of $M$ considered in \cite{Moon}, it does establish the fibration result claimed therein and also in this article.

Before I state the theorem, for ease of reference, I shall restate a certain technical condition regarding the 3-manifolds under consideration here and in \cite{Elkalla}, \cite{Moon}, and \cite{Sah}:\\

\textit{(A)  Let $U$ be a finitely generated subgroup of the fundamental group $G$ of a compact 3-manifold $M$.  Suppose that $U$ contains a nontrivial subnormal subgroup $N$ of $G$, and that the index $|G:U|$ is infinite.  If $N$ is not isomorphic to the infinite cyclic group $\mathbb{Z}$, then $M$ is finitely covered by a bundle over $\mathbb{S}^1$ with fiber a compact surface $F$ and $\pi_1(F)$ is commensurable with $U$.}\\

The promised result is the following theorem:

\begin{thm}\label{BoundaryCasesSettled}Let $M$ be a compact 3-manifold with $M=X_1\cup_{T}X_2$, or $M=X_1\cup_{T}$, where statement (A) holds for $X_i$, with $i=1,2$.  Suppose that $G=\pi_1(M)$ contains a finitely generated subgroup $U$ of infinite index satisfying the following:
\begin{enumerate}
\item $U$ contains a nontrivial subnormal subgroup $N$ of $G$,
\item $N$ intersects nontrivially the fundamental group of the splitting torus,
\item $N\cap\pi_1(X_i)$ is not isomorphic to $\mathbb{Z}$. 
\end{enumerate}
If the graph of groups $\mathcal{U}$ corresponding to $U$ is of finite diameter, then $\widehat{M}$ has a finite cover $\widetilde{M}$ which is a bundle over $\mathbb{S}^1$ with fiber a compact surface $F$, and $\pi_1(F)$ is commensurable with $U$.
\begin{proof}If $M$ is closed, the theorem is essentially a restatement of Theorem 2.9 in \cite{Moon}, with the trivial consideration that $N$ is subnormal rather than normal.  Therefore I now assume that $\partial M\neq\emptyset$.  The reader can verify for themselves that the proof of Theorem 2.9 in \cite{Moon} remains valid up to line 2 on page 31 under the hypothesis on $N$ and $M$ stated above when $\partial M\neq\emptyset$.  The only difference is that the surface $S'$ constructed in the proof of Theorem 2.9 is now a surface with nonempty boundary, and consequently the manifold $M''$, the compact core of $M'$, is itself a manifold with nonempty boundary.  While the remainder of the proof of Theorem 2.9 fails when $\partial M\neq\emptyset$, we can apply Theorem \ref{MainThm} after showing that we have an action of $\mathbb{Z}$ on the cover $M=S'\times\mathbb{R}$, which satisfies the hypothesis of the theorem above.  Using the notation employed in \cite{Moon}, we see that we have the tower of covers $S'\times\mathbb{R}\rightarrow M'\rightarrow M$.
	
We verify that the ingredients in the proof in \cite{Moon} are also present in the $\partial M\neq\emptyset$ case, although for different reasons.  We first show that $\pi_1(S')\neq\mathbb{Z}$, from which it will follow that $N_{\pi_1(S')}=\pi_1(M')$ is finitely generated.  

\textit{Case 1: $M=X_1\cup_{T} X_2$}

First, since $S'\times\mathbb{R}$ is an orientable cover of the orientable manifold $M$, it is itself orientable, which, in turn, implies that the submanifold $S'\times\left\lbrack -1,1\right\rbrack$ of $S'\times\mathbb{R}$ is orientable, hence $S'$ is itself orientable, by the remark on lines 7 and 6 from the bottom, on page 5 of \cite{Hatcher3M}.  We also observe that $S'$ was obtained by gluing finitely many copies of the surfaces $S_i$, which are finite covers of the fibers denoted in \cite{Moon} by $F_i$ chosen in such a way so that $\partial S_i$ has at least 2 boundary components. If either $S_i$ were a surface other than the annulus $\mathbb{S}^1\times I$, then $\pi_1(S_i)$ would be a free group of rank at least 2.  Since $\pi_1(S')$ has the structure of a graph of groups where the vertex groups are all isomorphic to $\pi_1(S_i)$, which must necessarily embed in $\pi_1(S')$ by Bass-Serre theory, we conclude that $\pi_1(S')$ would contain a nonabelian free group, which is clearly impossible.  Therefore, both $S_1$ and $S_2$ are the annulus $\mathbb{S}^1\times I$.  However, in this case, $X_i$ is virtually covered by an annulus bundle over $\mathbb{S}^1$, and thus $\pi_1(X_i)$ contains $\mathbb{Z}\rtimes\mathbb{Z}$ with finite index.  However, such a group cannot contain a (finitely generated) subgroup of infinite index $U$, which contains a nontrivial normal subgroup $N$ not isomorphic to the infinite cyclic group $\mathbb{Z}$, which is a contradiction.  Hence, again by Theorem 3.2 \cite{ScottNorm}, $N_{\pi_1(S')}$ is finitely generated.
As in the case without boundary, $\pi_1(T)\cap\pi_1(S')$ is isomorphic to $\mathbb{Z}$:  It is nontrivial as in \cite{Moon}, since the boundary components of the copies of $S_i$ represent elements of $\pi_1(S_i)$ which map to a nontrivial $\alpha^{r_1r_2}\in\pi_1(M)$; thus, $\pi_1(T)\cap\pi_1(S')$ is a nontrivial free group, which is a subgroup of a free abelian group.  We again conclude that as in \cite{Moon}, $|N_{\pi_1(S')}:\pi_1(S')|=\infty$.

\textit{Case 2: $M=X_1\cup_T$}

The argument in this case is analogous, as $S'$ remains a surface with nonempty boundary after the gluing of the boundary annuli of $F_1\times\mathbb{R}$ which map to the boundary tori $T_i$ in $Y_1$.  Therefore, again $\pi_1(S')=\mathbb{Z}$ implies that $S'$ is the annulus:  In this case, $S'$ was obtained by gluing boundary components of $F_1$; therefore, $\pi_1(S')=\left(\left(\left(\pi_1(F_1)*_{\mathbb{Z}} \right)*_{\mathbb{Z}}\right)...\right)*_{\mathbb{Z}}$, and $F_1$ must be the annulus by the same rank argument, hence because $\partial S'\neq\emptyset$, we have $F_1=S'$.  As in the case above, this implies a contradiction with the theorem's hypothesis.  Once again, $|N_{\pi_1(S')}:\pi_1(S')|=\infty$.

Now, applying Theorem 3 in \cite{HempelJaco}, we conclude that a finite cover of $\widehat{M''}$, the Poincaré associate of the core of $M'$, fibers over $\mathbb{S}^1$ with fiber a compact surface, and that $\pi_1(S')$ is subgroup of finite index of the fundamental group of the fiber; thus, $\pi_1(M')=\pi_1(M'')$ contains $\pi_1(S')$ as a normal subgroup, the quotient by which contains $\mathbb{Z}$ with finite index.

Hence, we see that the automorphism group $\mathfrak{G}$ of the cover $S'\times\mathbb{R}\rightarrow M'$, being isomorphic to the quotient of the normalizer of $\pi_1(S'\times\mathbb{R})$ in $\pi_1(M')$ by $\pi_1(S'\times\mathbb{R})$ - see Proposition 1.39(b) in \cite{Hatcher}, contains the infinite cyclic group $\mathbb{Z}$ with finite index.  

The manifold $M$ is compact, therefore, we can take its double $DM$, embed it in $\mathbb{R}^W$ for some $W\in\mathbb{N}$ by Whitney's Embedding Theorem, see page 53 in \cite{GuilleminPollack}, and pull back the Riemannian metric on $\mathbb{R}^W$ to obtain a Riemannian metric on $M$, which we can then pull back on $M'$, and then on $S'\times\mathbb{R}$, under the covering maps to obtain the desired Riemannian metric on $S'\times\mathbb{R}$ for which the action of $\mathfrak{G}$ is by isometries.

The final step we need to make, before we can apply Theorem \ref{MainThm}, is to show that for $S_0=S'\times\left\lbrace 0\right\rbrace$, we can find a positive integer $n_0$ with the property that $\gamma^n S_0\cap S_0=\emptyset$, for all $n\geq n_0$, where $\gamma$ is a generator of the infinite cyclic subgroup of the automorphism group $\frak{G}$.

 First, we produce an integer $n_0>0$ such that $\gamma^n S_0\cap S_0\neq S_0$, for all $n\geq n_0$.  Note that if for some $n$ we have $\gamma^nS_0\cap S_0=S_0$, then $\gamma^n S_0\supset S_0$, and if additionally $\gamma^n S_0\neq S_0$, then $S'\times\mathbb{R}-S_0$ will remain connected, as $S'\times\mathbb{R}-S_0\supset (S'\times\mathbb{R}-\gamma^n S_0) \cup\left\lbrace p \right\rbrace$, for some $p\in \gamma^n S_0-S_0$, and we see that the set on the right side of the containment is connected.  Therefore, for this $n$ we have $\gamma^n S_0=S_0$.  This cannot hold for infinitely many values of $n>0$, because $S_0$ is compact.  To show this, we argue by contradiction and we make the assumption that the set of $n$ such that $\gamma^n S_0=S_0$ is infinite.  Let $n_i$ be an infinite sequence of such values of $n$.   Then, since $\left\langle\gamma\right\rangle$ acts by covering transformations, given any $x_0\in S_0$, there is an open set $\mathcal{U}\subset S'\times\mathbb{R}$, such that if $\gamma^i\mathcal{U}\cap\gamma^j\mathcal{U}\neq\emptyset$, then $i=j$, see page 72 in \cite{Hatcher}.  But we can find, for the length metric induced by the Riemannian metric on $S'\times\mathbb{R}$ for which $\eta$ is an isometry, an open ball $B(x_0,\epsilon)\subset\mathcal{U}$, and we see that $d(\gamma^i x_0,\gamma^j x_0)>2\epsilon$, for any $i\neq j$.  Therefore, the sequence $\gamma^{n_i} x_0\in S_0$ can have no Cauchy subsequence and no accumulation point in $S_0$, which it must, by compactness of $S_0$.  Hence, there exists an integer $n_0>0$ such that for all $n>n_0$, $\gamma^n S_0\neq S_0$.  From now on, we let $n$ be an integer greater than $n_0$, so that $\gamma^n S_0\cap S_0\neq S_0$.  We proceed to show that we must also have $\gamma^n S_0\cap S_0=\emptyset$.
 
 Now, recall from \cite{Moon} that the cover $S'\times\mathbb{R}$ is obtained by gluing copies of  $S_i\times\mathbb{R}$ along their boundary annuli which are of the form $C_i\times\mathbb{R}$, where $C_i$ is a component of $\partial S_i$ which projects under $p_1\circ p_2$ and $q_1\circ q_2$, respectively, into $T$.  The covering map $\eta: S'\times\mathbb{R}\rightarrow M$ is described by pasting together the various copies of the covering maps $p_1\circ p_2\circ p_3$ and $q_1\circ q_2\circ q_3$ to $X_1$ and $X_2$ from the copies of $S_i\times\mathbb{R}$.  If $\gamma^n S_0\cap S_0\neq\emptyset$, then since $S_0$ and $\gamma^n S_0$ are closed sets, we can find a non-manifold point $x_0\in\gamma^n S_0\cap S_0$ for the union of the two surfaces $\gamma^n S_0\cup S_0\subset S'\times\mathbb{R}$, which under the local diffeomorphism $\eta$ would map to the non-manifold point $\eta(x_0)$ of $\eta((\gamma^n S_0\cup S_0)\cap\eta^{-1}(X_i)))\subset X_i$.  Hence, the image of $x_0$ under the appropriate copy of $p_3\circ p_2$, or $q_3\circ q_2$, must be a non-manifold point of a copy of the fiber $F_i$ in $Y_i$ for $i=1$ or $i=2$, as both $\gamma^n S_0\cap\eta^{-1}(X_i)$ and  $S_0\cap\eta^{-1}(X_i)$ cover the same embedded copy of the fiber $F_i$ in $Y_i$, and $Y_i$ covers $X_i$ by a local diffeomorphism.  However, all of the fibers $F_i$ of the bundle $Y_i$ are embedded surfaces and have no non-manifold points, which is a contradiction.
 
 Now, by Theorem \ref{MainThm}, the cover $M'$ is compact, since the automorphism group of the cover $\mathfrak{G}$ acts with compact quotient.  Now, we see that by Theorem 3 in \cite{HempelJaco}, a finite cover $V$ of $\widehat{M'}$, and therefore also of $\widehat{M}$, fibers over the circle with fiber a compact surface $F$.  We now show that either $U$ is commensurable with $\pi_1(F)$, or a finite cover of $V$, hence also of $M$, fibers over the circle with fiber a compact surface $F'$, and $\pi_1(F')$ is commensurable with $U$; thus proving, in both cases, that a finite cover of $M$ fibers in the desired way.  To this end, consider $U\cap\pi_1(F)<\pi_1(M)$.  We note that we have the short exact sequence $1\rightarrow\pi_1(F)\rightarrow \pi_1(V)\rightarrow \mathbb{Z}\rightarrow 1$, and that $\pi_1(V)$ is a subgroup of finite index of $\pi_1(M)$.  If $U\cap\pi_1(F)$ is of finite index in both $U$ and $\pi_1(F)$, then we are done - $U$ is commensurable with $\pi_1(F)$.  Note that if $\pi_1(F)\cap N=\left\lbrace 1\right\rbrace$, then $\pi_1(V)\cap N$ embeds in the infinite cyclic group $\mathbb{Z}$, which means that $N$ contains $\mathbb{Z}$ with finite index.  In this case, $N$ is a nontrivial finitely generated subnormal subgroup of $G$, hence by Corollary 2.3 in \cite{Elkalla}, $N$ is the fundamental group of a compact surface - see also Theorem 1 in \cite{HempelJaco}.  Let $\Sigma$ be such a surface.  We see that $\partial\Sigma=\emptyset$ is ruled out by the classification of surfaces, whereas if $\partial\Sigma\neq\emptyset$, then as before, we see, after passing to the orientable double cover if necessary, that the finite cover of $\Sigma$ whose fundamental group is $\mathbb{Z}$ is the annulus $\mathbb{S}^1\times I$, hence $\Sigma$ is the annulus or the Moebius strip, and in either case $N=\mathbb{Z}$, contradicting the hypothesis of the theorem.  Thus, we shall henceforth assume that $\pi_1(F)\cap N\neq\left\lbrace 1\right\rbrace$.

 We consider the remaining cases separately below:
 
 \textit{$U\cap\pi_1(F)$ is of finite index in $\pi_1(F)$ but not in $U$:}  This is impossible, since if $\pi_1(F)$ were virtually contained in $U$, but $U$ were not virtually contained in $\pi_1(F)$, then $U$ would be of finite index in $\pi_1(M)$, which contradicts the theorem's hypothesis.  Here are the details:  Consider the subgroup $(U\cap\pi_1(V))\pi_1(F)$ of $\pi_1(V)$.  The map $(U\cap\pi_1(V))f\rightarrow U\cap\pi_1(F)f$, $f\in\pi_1(F)$, from the right cosets of $U\cap\pi_1(V)$ in $(U\cap\pi_1(V))\pi_1(F)$ to the set of right cosets of $U\cap\pi_1(F)$ in $\pi_1(F)$ is injective, and therefore $|(U\cap\pi_1(V))\pi_1(F):U\cap\pi_1(V)|<\infty$.  Now, consider the epimorphism $\phi:\pi_1(V)\rightarrow\mathbb{Z}$ and its restriction $\phi|_{U\cap\pi_1(V)}:U\cap\pi_1(V)\rightarrow\mathbb{Z}$.  Since $|U:U\cap\pi_1(V)|<\infty$ and $|U:U\cap\pi_1(F)|=\infty$, we must have $|U\cap\pi_1(V):U\cap\pi_1(F)|=\infty$.  Since now $Ker(\phi_{U\cap\pi_1(V)})=U\cap\pi_1(F)$ is of infinite index in $U\cap\pi_1(V)$, we have $\overline{t}^k\in Im(\phi|_{U\cap\pi_1(V)})$, for some $k>0$.  Thus $t^kf\in U\cap\pi_1(V)$, for some $f\in\pi_1(F)$, and $t^k\in (U\cap\pi_1(V))\pi_1(F)$.  Now, $(U\cap\pi_1(V))\pi_1(F)t^i\supset\pi_1(F)t^i$, and since $\pi_1(V)=\bigcup_i\pi_1(F)t^i$, we see that $\pi_1(V)=\bigcup_i (U\cap\pi_1(V))\pi_1(F)t^i=\bigcup_{i=0}^{k-1}(U\cap\pi_1(V))\pi_1(F)t^i$.  This shows that $|\pi_1(V):(U\cap\pi_1(V))\pi_1(F)|<\infty$, thus $|\pi_1(V):U\cap\pi_1(V)|<\infty$; from here, we get $|\pi_1(M):U\cap\pi_1(V)|<\infty$, and finally $|\pi_1(M):U|<\infty$ - a contradiction, as claimed.

 \textit{$U\cap\pi_1(F)$ is of finite index in $U$, but not in $\pi_1(F)$:}  In this case, $U\cap\pi_1(F)$ is a finitely generated subgroup of $\pi_1(F)$ of infinite index, which contains a nontrivial subnormal subgroup $N\cap\pi_1(F)\triangleleft_s\pi_1(F)$.  By Theorem 2.1 in \cite{Sah}, $F$ is the torus, since $V$ and therefore $F$ is orientable.  However, torus bundles over $\mathbb{S}^1$ are geometric, hence $V$ is a geometric manifold.  Since $U\cap\pi_1(V)$ is of finite index in $U$ and of infinite index in $\pi_1(V)$, it is a finitely generated infinite index subgroup of $\pi_1(V)$, whereas the subgroup $N\cap\pi_1(V)<U\cap\pi_1(V)$ is a nontrivial subnormal subgroup of $\pi_1(V)$ which is not infinite cyclic.  Now, we can apply Theorem 3.9 in \cite{Sah} to $V$, $U\cap\pi_1(V)$, and $N\cap\pi_1(V)$, to conclude that a finite cover of $V$, hence also of $M$, is a bundle over $\mathbb{S}^1$, with fiber a compact surface $F'$, and also that $U\cap\pi_1(V)$ is commensurable with $\pi_1(F')$.  Since $U\cap\pi_1(V)$ is of finite index in $U$, we conclude that $U$ is itself commensurable with $\pi_1(F')$, as desired.
 
 \textit{$U\cap\pi_1(F)$ is of finite index in neither $\pi_1(F)$ nor $U$:}  If $U\cap\pi_1(F)$ is finitely generated, then as in the case above, the conclusion follows from Theorem 3.9 in \cite{Sah} for the same reasons.  Therefore, we assume that the 3-manifold group  $U\cap\pi_1(F)$ is not finitely generated.  In such a case, however, Theorem 2.6 in \cite{Elkalla} implies that the nontrivial subnormal subgroup $\pi_1(F)\cap N$ of $U\cap\pi_1(F)$ is infinite cyclic and also that $\pi_1(F)\cap N\triangleleft U\cap\pi_1(F)$.  Since the non-finitely generated group $U\cap\pi_1(F)$ properly contains its infinite cyclic subgroup $\pi_1(F)\cap N$, we can find an element $t$ of $U\cap\pi_1(F)$ such that $t\notin\pi_1(F)\cap N$.  Consider now the subgroup $H=\left\langle \alpha,t\right\rangle$ of $\pi_1(F)$, where $\alpha$ is a generator of $\pi_1(F)\cap N=\mathbb{Z}$. Since $H<\pi_1(F)$ contains $\mathbb{Z}^2$ with index at most 2, we see that $F$ is an orientable surface covered by the torus.  The cover is necessarily a finite cover by compactness, hence $F$ itself is the torus, and $\pi_1(F)=\mathbb{Z}^2$.  This, however, contradicts the assumption that $U\cap\pi_1(F)$ is not finitely generated, as $\mathbb{Z}^2$ has no non-finitely generated subgroups, which finishes the proof.
 
 Finally, we note that the proof that $U$ is commensurable with $\pi_1(F)$, for a certain virual fibering of $M$, is also applicable to the case when $M$ is closed.  This addresses the lack of proof of this fact in \cite{Moon}.
\end{proof}
\end{thm}

\end{section}

\begin{section}{Irreducibility of the summands in an irreducible torus sum\\Irreducibility of $\widehat{M}$}\label{IrredSec}
	
The idea of making use of the $\pi_1$-injectivity of the embedded torus $\mathcal{T}$ is a key ingredient in the following lemma and is due to Peter Scott.  After having submitted this pre-print for peer review, the anonymous referree stated that this result was already known.  I myself know of no references for it, which is, of course, not a proof of their absence, but I do not doubt the referee's words.  I do feel, however, that this argument may still be of independent interest.

\begin{lem}\label{A}Let $M=M_1\cup_{\mathcal{T}} M_2$, or $M=M_1\cup_{\mathcal{T}}$, where $\mathcal{T}\subset M$ is an incompressible torus.  If $M$ is irreducible, then any 2-sphere $S\subset M_i-\mathcal{T}$ bounds a 3-ball in $M_i-\mathcal{T}$.
	
\begin{proof}By abuse of notation, let $S$ be a 2-sphere embedded in $M_i-\mathcal{T}$, for $i=1$ or $2$, and $B$ the 3-ball that it bounds in $M$.  We need to show that $B$ is contained in $M_i-\mathcal{T}$.  Consider $B^0=B-\partial{B}$, and consider also $B^0\cap\mathcal{T}$.  Clearly, $B^0\cap\mathcal{T}$ is open in $\mathcal{T}$ as $B^0$ is open in $M$.  On the other hand, since $S$ and $\mathcal{T}$ are compact, $\inf\left\lbrace d(x,y):x\in S, y\in\mathcal{T}\right\rbrace > 0$, for any length metric on $M$ induced by a Riemannian metric.  Therefore, if $\left\lbrace p_i \right\rbrace$ is a sequence in $B^0\cap\mathcal{T}$ such that $p_i\rightarrow p$, we must have $p\in B$ as $B$ is sequentially compact, but $p\notin S$, hence $p\in B^0$; $\mathcal{T}$ is also compact, therefore sequentially compact, hence $p\in\mathcal{T}$, therefore $p\in B^0\cap\mathcal{T}$.  Thus, we see that $B^0\cap\mathcal{T}$ is also closed.  Because $\mathcal{T}$ is connected, $B^0\cap\mathcal{T}$ is either empty or equals $\mathcal{T}$.  We cannot have $B^0\cap\mathcal{T}=\mathcal{T}$ because then $\mathcal{T}\subset B^0$, which would imply that $\mathcal{T}$ is contractible in $M$, thus contradicting the $\pi_1$-injectivity of $\mathcal{T}$.  Thus, we have $B^0\cap\mathcal{T}=\emptyset$, hence $B$ is contained in a connected component of $M-\mathcal{T}$.  Since $\partial B\subset M_i$, we must also have $B\subset M_i-\mathcal{T}$, thus proving the claim.  The case $M_1\cup_{\mathcal{T}}$ is identical - one only needs to consider both tori.
\end{proof}
\end{lem}

\begin{pro}\label{B}Let $M=M_1\cup_{\mathcal{T}} M_2$, or $M=M_1\cup_{\mathcal{T}}$, where $\mathcal{T}\subset M$ is an incompressible torus.  If $M$ is irreducible, then $M_i$ is irreducible.

\begin{proof}Let $f:S\rightarrow M_i$ be an embedding of the 2-sphere into $M_i$, for $i=1$ or $2$.  If $f(S)\cap\mathcal{T}=\emptyset$, Lemma \ref{A} shows that $f(S)$ bounds a 3-ball in $M_i-\mathcal{T}$, hence also in $M_i$ thus proving the claim. 
	
Suppose, therefore, that $f(S)\cap\mathcal{T}\neq\emptyset$.  Let $N$ be a collared neighborhood of $\mathcal{T}\subset\partial M_i$ homeomorphic to $\mathcal{T}\times \left\lbrack 0,1\right\rbrack$, and let $\mathcal{T}$ be identified with $\mathcal{T}\times\left\lbrace 1\right\rbrace\subset N$ under the homeomorphism.  Will show that $f(S)$ bounds a 3-ball by first homotoping $f(S)$ to be disjoint from $\mathcal{T}\subset\partial M_i$, and then using the previous result to show that $f(S)$ is inessential.  To this end, let $F$ be the map $F((x,t),s)=(x,t(1-s/2))$, for $(x,t)\in N$, with $x\in\mathcal{T}$ and $t\in\left\lbrack 0,1\right\rbrack$, so that $F:N\times I\rightarrow N$.  Note that $F$ restricts to the identity on $\mathcal{T}\times\left\lbrace 0\right\rbrace\subset M_i^0$, hence extends to a continuous homotopy $M_i\times\left\lbrack 0,1\right\rbrack\rightarrow M_i$ defined to be the identity off $N$, which we will still denote by $F$, and note also that for this extension we have $F(\cdot,0)=id_{M_i}(\cdot)$.  Note, further, that $F(f(\cdot),0)=f(\cdot)$, and if we set $F(f(\cdot),1)=f_1(\cdot)$, then $f_1(p)=(x(p),t(p)/2)$ where $f(p)=(x(p),t(p))$ whenever $f(p)\in N$.  It is obvious that $f_1$ is still an embedding of $S$ into $M_i$ since $f_1$ is injective, as clearly $f_1(p_1)=f_1(p_2)$ is impossible in either of the two cases: $f_1(p_1), f_1(p_2)\in N$ or $f_1(p_1), f_1(p_2)\notin N$.  We also have $f_1(S)\cap\mathcal{T}=\emptyset$.  By Lemma \ref{A}, $f_1(S)$ bounds a 3-ball.  Let $\phi:B\rightarrow M_i$ be this embedding of the 3-ball into $M_i$.  Since $\phi|_{\partial B}(p)=f_1(p)=F(f(p),1)$, for $p\in S=\partial B$, the map $\phi:B\rightarrow M_i$ and the map $S\times\left\lbrack 0,1\right\rbrack\rightarrow M_i$, defined by $(p,t)\rightarrow F(f(p)),t)$, fit together to give a continuous map from $\overline{\phi}:B\cup_{\partial B\cong S\times\left\lbrace 1\right\rbrace} \left( S\times\left\lbrack 0,1\right\rbrack\right)\cong B \rightarrow M_i$.  Thus, we see that the emedded 2-sphere $f(S)=\overline{\phi}(S\times\left\lbrace 0 \right\rbrace)$ is inessential in $M_i$.  This happens if and only if $f(S)$ bounds a compact simply connected submanifold $W$ of $M_i$, see Lemma 9.1.8 in \cite{TheDraft}.  Now, since $H_1(W)=\pi_{1,ab}(W)$, we see that $H_1(W)=0$ and therefore $f(S)$ is null homologous in $W$.  By considering the 3-cycles in $W$ whose boundary is precisely $f(S)$, we see that each 3-simplex of $W$ must be included in such a 3-cycle, as we cannot have any free faces not in $f(S)$, and we further observe that $f(S)$ is the only boundary component of $W$.  Now, after attaching a 3-ball to $W$ along the boundary sphere and applying Van Kampen's Theorem, we see that we have a simply connected closed 3-manifold, which must be $\mathbb{S}^3$ by the positive solution to the Poincaré Conjecture.  Thus, we conclude that $W$ is itself a 3-ball, which concludes the proof.
\end{proof}
	
\end{pro}
The following result makes use of certain elementary methods from homology which can all be found in \cite{Hatcher}.  I am to grateful to my late Ph.D. adviser Peter Scott for communicating to me, via e-mail, the short homological argument showing that the boundary of a simply connected compact 3-manifold, whenever nonempty, consists of 2-spheres - Peter did this at a time when many others would not, could not have made the effort or the time.  The reader can compare this proof with the proof I gave in \cite{SahErr}, which requires the added hypothesis that the boundary of $M$ be either empty or toroidal.
\begin{pro}\label{irredOfMh}Let $M$ be a compact connected orientable 3-manifold.  Suppose that $N$ is a non-trivial subnormal subgroup of $G=\pi_1(M)$, and suppose that $N$ is contained in a finitely generated subgroup $U<G$ whose index in $G$ is infinite.  Then, $M$ is either irreducible, or it can be expressed as the connected sum of an irreducible compact manifold $M_1$ with finitely many 3-balls.  In particular, the Poincare associate $\widehat{M}$ is irreducible, and if the boundary of $M$ contains no 2-spheres, then $M$ is itself irreducible.
\begin{proof}  Consider the decomposition of $M$ into prime summands: $M\cong M_1\# M_2\#...\# M_p$, where each $M_i$ is a prime manifold.  We have $G=G_1*G_2*...*G_p$, where $G_i=\pi_1(M_i)$.  Since by Theorem 1.5 in \cite{Elkalla} $G$ must be indecomposable with respect to free products of groups, we must have $\pi_1(M_i)=\left\lbrace 1\right\rbrace$, for all but one value of $i$, hence without loss of generality, we assume that $\pi_1(M_i)=\left\lbrace 1\right\rbrace$ for all $i>1$.  Now, consider the simply connected summands $M_i$ with nonempty boundaries $\partial M_i$, if there are any.  For each such summand, from the long exact homology sequence of the pair $(M_i, \partial M_i)$, we have $H_2(M_i, \partial M_i)\rightarrow H_1(\partial M_i)\rightarrow H_1(M_i)$.  Using Lefschetz duality, we obtain $H_2(M_i,\partial M_i;\mathbb{Z})\cong H^1(M;\mathbb{Z})$, while $\pi_1(M_i)=\left\lbrace 1\right\rbrace$ implies that $H_1(M_i)=0$.  To compute $H^1(M_i;\mathbb{Z})$ we use the Universal Coefficients Theorem for cohomology; we have: $0\rightarrow Ext^1_{\mathbb{Z}}(H_0(M_i;\mathbb{Z}),\mathbb{Z})\rightarrow H^1(M_i;\mathbb{Z})\rightarrow Hom_{\mathbb{Z}}(H_1(M_i;\mathbb{Z}),\mathbb{Z})\rightarrow 0$.  Since $H_0(M_i;\mathbb{Z})=\mathbb{Z}$, $Ext^1_{\mathbb{Z}}(H_0(M_i;\mathbb{Z}), \mathbb{Z})=0$, and since $H_1(M_i)=0$, $Hom_{\mathbb{Z}}(H_1(M_i;\mathbb{Z}),\mathbb{Z})=0$.  Thus, we have $0\rightarrow H^1(M_i;\mathbb{Z})\rightarrow 0$, from where we conclude that $H^1(M_i;\mathbb{Z})=0$ and $H_2(M_i,\partial M_i;\mathbb{Z})=0$, hence finally $H_1(\partial M_i)=0$.  This means that if the orientable surface $\partial M_i$ is nonempty, it is a union of 2-spheres.  However, $M_i$ is prime and simply connected, therefore, $M_i$ must be also irreducible, hence $M_i$ is a 3-ball and $\partial M_i=\mathbb{S}^2$.  If $\partial M_i=\emptyset$, then by the Poincare Conjecture $M_i=\mathbb{S}^3$.  Now, $M$ is expressed as the connected sum of $M_1$ with 3-balls and 3-spheres.  On the other hand, $M_1$ is also prime, hence it is either equal to $\mathbb{S}^2\times\mathbb{S}^1$, which is impossible since in this case $\pi_1(M)=\pi_1(M_1)=\mathbb{Z}$ cannot satisfy the hypothesis of the proposition, or $M_1$ is irreducible and its boundary cannot contain any 2-spheres, otherwise we would conclude that $M_1$ is also a 3-ball, which is also impossible.  Now, it is clear that $\widehat{M}\cong M_1$ as every boundary 2-sphere in $M$ was the boundary of the 3-ball $M_i$.  Therefore, $\widehat{M}$ is irreducible as claimed. 
\end{proof}
\end{pro}

\end{section}

We record here an immediate corollary to Theorem 5.3 in \cite{Sah}, as corrected in \cite{SahErr}:

\begin{section}{An application to fiberings of compact 3-manifolds over $\mathbb{S}^1$}

As an application of Theorem \ref{MainThm}, I give the following result which generalizes Theorem 2.10 in \cite{Moon}, and which is now vindicated as a true statement of mathematical fact, see \cite{SahErr}.  The reader should know that this generalization was only possible after the proof of the Geometrization Theorem by Perelman in 2003.  In light of Proposition \ref{B}, this result, which first appeared in \cite{Sah} and \cite{SahErr}, now becomes:

\begin{thm}Let $M$ be a closed 3-manifold.  If $G=\pi_1(M)$ contains a finitely generated subgroup $U$ of infinite index in $G$ which contains a nontrivial, subnormal subgroup $N$ of $G$, then: (a) $M$ is irreducible, (b) if further: 
	\begin{enumerate}
		\item  $N$ has a subnormal series of length $n$ in which $n-1$ terms are assumed to be finitely generated, and
		\item either all inclusions $N_i \hookrightarrow N_{i+1}$, for $i>0$, are of finite index, or there exist (at least two) indices $i_0$ and $i_1$, $i_0\neq i_1$, $i_0,i_1>1$, such that the inclusions $N_i \hookrightarrow N_{i+1}$ are of infinite index for $i=i_0, i_1$, or $N=N_0$ is finitely generated and there exists (at least) one value of the index $i$ for which the inclusion $N_{i}\hookrightarrow N_{i+1}$ is of infinite index, and
		\item $N$ intersects nontrivially the fundamental groups of the splitting tori of some decomposition $\mathfrak{D}$ of $M$ into geometric pieces, and
		$N\cap\pi_1(X_i)\neq\mathbb{Z}$ for all geometric pieces $X_i\in\mathfrak{D}$,
	\end{enumerate}
	then, $M$ has a finite cover which is a bundle over $\mathbb{S}$ with fiber a compact surface $F$ such that $\pi_1(F)$ and $U$ are commensurable.  Further, any decomposition of $M$ along a collection of incompressible tori yields pieces which are irreducible, in particular, each $X_i\in\mathfrak{D}$ is irreducible. 
\end{thm}
\begin{proof}The proof of the theorem is essentially contained in \cite{Sah} and \cite{SahErr}.  The corrigendum \cite{SahErr} notes that, as of the time of its publication, the theorem was contingent on the proof of a result at least as general as Theorem \ref{BoundaryCasesSettled} for compact manifolds with nonempty boundary, due to the absence of such a proof from \cite{Moon}, which I had somehow overlooked at the time.  Therefore, the only statement which needs proof is the claim about the irreducibility of the $X_i$; this, however, follows immediately from Proposition \ref{B} and induction.
\end{proof}

An easier to state special case of the preceding result is the following:

\begin{cor}Let $M$ be a compact 3-manifold with empty or toroidal boundary.  If $G=\pi_1(M)$ contains a finitely generated subgroup $U$ of infinite index in $G$ which contains a nontrivial, normal subgroup $N$ of $G$, then: (a) $M$ is irreducible, (b) if further: 
	
\begin{enumerate}
\item $N$ intersects nontrivially the fundamental groups of the splitting tori of some decomposition $\mathfrak{D}$ of $M$ into geometric pieces, and
\item $N\cap\pi_1(X_i)\neq\mathbb{Z}$ for all geometric pieces $X_i\in\mathfrak{D}$,
\end{enumerate}

then, $M$ has a finite cover which is a bundle over $\mathbb{S}$ with fiber a compact surface $F$ such that $\pi_1(F)$ and $U$ are commensurable, and each $X_i$ is an irreducible, compact, geometric 3-manifold.
\end{cor}

\end{section}

\begin{section}{Acknowledgements}
I would like to posthumously express my sincerest gratitude to the late Prof. Peter Scott for the many years of patient advice and his friendship, and to thank his son David for allowing me, on behalf of Prof. Scott's estate, to use his unpublished notes on 3-manifolds.  I would also like to acknowledge and express my heartfelt thanks for the very timely and very helpful advice Chris Mooney gave me during that brief and fateful conversation we had back in 2011, which resulted in Lemmas 3.1 and 3.2 in my very first paper published in the Journal of Groups, Complexity, Cryptography in 2015, and to Enric Ventura, who brought to my attention the immediate converse of the main result, see Proposition 3.7, in this older article.  I am also very thankful to Moon Duchin for introducing me to Chris Mooney and for giving my first paper's initial draft a first read through at a time when nothing seemed to be going according to plan for me: stuck on two thesis problems, with my time to finish my degree nearly having expired.
I would also like to thank the Elmhurst Public Library of the City of Elmhurst, Illinois for allowing me to use its quiet area to read and work in, and the Founders Memorial Library of Northern Illinois University for extending me courtesy borrowing privileges, which have been invaluable.  Finally, I would like to thank all the mathematicians who post useful bits of mathematical advice and references on mathoverflow.net.

\end{section}

\bibliographystyle{plain}

\end{document}